\documentclass{article}
\usepackage{maa-monthly}
\usepackage{hyperref}
\usepackage{physics}
\usepackage{enumitem}
\usepackage{float}

\makeatletter
\renewcommand{\@biblabel}[1]{[#1]}						% Brackets in Reference-list
\makeatother

\newtheorem{theorem}{Theorem}
\newtheorem{lemma}[theorem]{Lemma}
\newtheorem{prop}[theorem]{Proposition}
\newtheorem{corollary}[theorem]{Corollary}

\theoremstyle{definition}

\newtheorem{observation}[theorem]{Observation}
\newtheorem{remark}[theorem]{Remark}

\newcommand{\diag}[1]{\operatorname{diag}\left( #1 \right)}		% diagonal matrix
\newcommand{\bb}[1]{\mathbb{#1}}						% Blackboard boldface for R, C, N, Q
\newcommand{\hyp}[2]{#1 \hyperref[#2]{\ref{#2}}}			% Hyperlink
\newcommand{\conv}[1]{\operatorname{conv}\left( #1 \right)}		% Convex Hull

\begin{document}
%-----------------------------------------------------------------------------------------------------------------------------------------------------------------------------------------------------------------------------
% Front-matter
%-----------------------------------------------------------------------------------------------------------------------------------------------------------------------------------------------------------------------------

\title{Matricial Proofs of Some Classical Results about Critical Point Location}
\markright{Matricial Proofs of Some Classical Results}
\author{Charles R.~Johnson and Pietro Paparella}

%    \subjclass is required.
%\subjclass[2010]{Primary 30C15 26C10; Secondary: 15A60 15A18 15B99}

\maketitle

%-----------------------------------------------------------------------------------------------------------------------------------------------------------------------------------------------------------------------------
%    	Abstract 
\begin{abstract}
The Gauss--Lucas and B\^{o}cher--Grace--Marden theorems are classical results in the geometry of polynomials.  Proofs of the these results are available in the literature, but the approaches are seemingly different. In this work, we show that these theorems can be proven in a unified theoretical framework utilizing matrix analysis (in particular, using the field of values and the differentiator of a matrix). In addition, we provide a useful variant of a well-known result due to Siebeck.
\end{abstract}

%-----------------------------------------------------------------------------------------------------------------------------------------------------------------------------------------------------------------------------
\section{Introduction.}
%----------------------------------------------------------------------------------------------------------------------------------------------------------------------------------------------------------------------------- 

The Gauss--Lucas and B\^{o}cher--Grace--Marden theorems are classical results that geometrically relate the location of the critical points of a polynomial to its zeros \cite{m1966,p2010}. (Kalman \cite{k2008b} attributes the latter result to Marden, who himself attributes the result to B\^{o}cher and Grace \cite{m1966}. Prasolov \cite[Theorem 1.2.3]{p2010} attributes the result to van den Berg \cite{vdb1888}.) There are several proofs of these statements in the literature, some either quite involved or incomplete (Kalman \cite{k2008b} notes that both the proofs of B\^{o}cher--Grace--Marden by Marden \cite{m1945,m1966} and B\^{o}cher \cite{b1892} are incomplete). Furthermore, the elementary approaches to these theorems in the literature are rather different. 

In recent decades there has been a realization, through several works, of how matrix-analytical ideas may be used to give insightful, and often simpler, proofs of facts about the zeros or critical points of polynomials (\cite[Chapters 5 and 6]{hj2013}, \cite{p2017, p2003, w1961}). Here, we continue the development of the use of matricial ideas in polynomials, using other important tools: the field of values of a matrix and the differentiator of a matrix. In particular, we give new and unified proofs of Gauss--Lucas, B\^{o}cher--Grace--Marden, and a variant of Siebeck that are quite brief relative to some expositions \cite{b2014, b1892, b2017, d2013, k2008b, m1945, m1966,mp2008}.

%-----------------------------------------------------------------------------------------------------------------------------------------------------------------------------------------------------------------------------
\section{Notation And Background.}
%----------------------------------------------------------------------------------------------------------------------------------------------------------------------------------------------------------------------------- 

For \( n \in \mathbb{N}\), let \(\langle n \rangle\) denote the set \(\{1, \dots, n\}\). We let \( \conv{S} \) denote the \emph{convex hull} of a subset \( S \) of \( \mathbb{C} \).

For positive integers \(m\) and \(n\), let \(M_{m, n}(\mathbb{F})\) denote the set of \(m\text{-by-}n\) matrices with entries from a field \(\mathbb{F}\). In the case when \( m = n\), \(M_{m, n}(\mathbb{F})\) is abbreviated to \(M_n(\mathbb{F})\).

For \(A\in M_n(\mathbb{F}) \), we let \( \sigma(A) \) denote the \emph{spectrum} (i.e., multiset of eigenvalues) of \(A\);  \(A_{(i)}\) denote the \(i\){th} \emph{principal submatrix} of \(A\), i.e.,  \(A_{(i)}\) is the \((n-1)\)-by-\((n-1)\) matrix obtained by deleting the \(i\){th} row and \(i\){th} column of \(A\); and \(\tau(A)\) denote the \emph{normalized trace of \(A\)}, i.e., \( \tau(A):= (1/n) \trace A\). For \( x \in \bb{F}^n\) and \( i \in \langle n \rangle \), \( x_{(i)} \in \bb{F}^{n-1} \) denotes the vector obtained by deleting the \(i\){th} entry of \( x \). Let \(\diag{\lambda_1,\dots,\lambda_n}\) denote the \emph{diagonal matrix} whose \(i\){th} diagonal entry is \( \lambda_i\).

An $n$-by-$n$ matrix \(H\) is called a \emph{Hadamard matrix (of order $n$)} if \(h_{ij} \in \{ \pm 1 \}\) and \(HH^\top = nI\). A matrix \(H\) is called a \emph{complex Hadamard matrix (of order $n$)} if \(|h_{ij}| = 1\) and \(HH^* = nI\). Notice that for any complex Hadamard matrix \(H \in M_n({\mathbb{C}})\), the matrix \(U := \frac{1}{\sqrt{n}}H\) is unitary, i.e., \(U^*U = UU^* = I\). For a fixed positive integer \(n\) and the complex scalar \(\omega := \exp(-2\pi i/n)\), the matrix 
\[
 F = F_n :=
\begin{bmatrix}
1 & 1 & 1 & \dots & 1 \\
1 & \omega & \omega^2 & \dots & \omega^{n-1} \\
1 & \omega^2 & \omega^4 & \dots & \omega^{2(n-1)} \\
      \vdots & \vdots & \vdots & \ddots & \vdots \\
      1 & \omega^{n-1} & \omega^{2(n-1)} & \dots & \omega^{(n-1)^2}
    \end{bmatrix},
\]
is called the \emph{discrete Fourier transform (DFT) matrix (of order $n$)}. As is well known, and otherwise easy to establish, DFT matrices are complex Hadamard matrices.

%-----------------------------------------------------------------------------------------------------------------------------------------------------------------------------------------------------------------------------
\section{Differentiators and Trace Vectors.}
%----------------------------------------------------------------------------------------------------------------------------------------------------------------------------------------------------------------------------- 

In \cite{p2003}, differentiators were studied and the concept of a trace vector was introduced to resolve several outstanding conjectures in the geometry of polynomials. 

For \(A \in M_n(\mathbb{C})\) and a unit vector \(z \in \mathbb{C}^n\), let \(P = P(z) := I - zz^* \in M_n(\mathbb{C})\).  If \(B := PAP|_{P\mathbb{C}^n}\) (the matrix \(B\) is called the \emph{compression of \(A\) onto \(P\mathbb{C}^n\)}), then \(P\) is called a \emph{differentiator (of \(A\))} if
\[ p_B(t) = \frac{1}{n} p'_A(t), \] 
in which \(p_M\) denotes the characteristic polynomial of \(M \in M_n(\mathbb{C})\).

If \(A \in M_n(\mathbb{C})\) and \(z\in \mathbb{C}^n\), then \(z\) is called a \emph{trace vector (for \(A\))} if \(z^*A^kz = \tau(A^k)\), for every nonnegative integer \(k\). Because of the case \(k=0\), it is clear that all trace vectors have unit length. 

In \cite[Theorem 2.5]{p2003}, it was shown that if \(A \in M_n(\mathbb{C})\) and \( P = I - zz^*\), then \(P\) is a differentiator of \(A\) if and only if \(z\) is a trace vector for \(A\). In addition, it was shown that every square matrix possesses at least one trace vector, i.e., every matrix possesses at least one differentiator \cite[Theorem 2.10]{p2003}. 

%-----------------------------------------------------------------------------------------------------------------------------------------------------------------------------------------------------------------------------
\begin{observation}
\label{obs:cancomps}
If \(e_i\) is a trace vector for \(A \in M_n(\mathbb{C})\), where \(e_i\) (\(1 \le i \le n\)) denotes the $i$th canonical basis vector of \(\mathbb{C}^n\), and \(P = I - e_ie_i^\top\), then the compression of \(A\) onto \(P\mathbb{C}^n\) is \(A_{(i)}\).
\end{observation}

The following result was given in \cite{hmpt2018}. 

%-----------------------------------------------------------------------------------------------------------------------------------------------------------------------------------------------------------------------------
\begin{lemma}
\label{lem:diagtrace}
Let \(D \in M_n(\mathbb{C})\) be a diagonal matrix and \(z \in \mathbb{C}^n\). If \(|z_i|=\frac{1}{\sqrt n}\) for every \(i \in \langle n \rangle\), then \(z\) is a trace vector for \(D\). 
\end{lemma}

%-----------------------------------------------------------------------------------------------------------------------------------------------------------------------------------------------------------------------------
\begin{theorem}
\label{thm:main}
Let $p$ be a polynomial of degree $n$ with zeros $\lambda_1,\dots, \lambda_n$ (including multiplicities) and critical points $\mu_1,\dots,\mu_{n-1}$ (including multiplicities). If \(H \in M_n(\mathbb{C})\) is a complex Hadamard matrix, \(D := \diag{\lambda_1, \dots, \lambda_n}\), and \(A := HDH^{-1}\), then \(\sigma \left(A_{(i)} \right) = \{ \mu_1,\dots,\mu_{n-1} \} \), for every \( i \in \langle n \rangle \).
\end{theorem}

\begin{proof}
Let \( i \in \langle n \rangle \). If \(U := \frac{1}{\sqrt{n}}H\), then \(A = HDH^{-1} = UDU^{-1} = UDU^*\). Notice that \(U^* e_i\) is the entrywise complex conjugate of the \(i\){th} row of \(U\). Since each element of this row has modulus \(\frac{1}{\sqrt{n}}\), it follows from \hyp{Lemma}{lem:diagtrace} that \(U^* e_i\) is a trace vector for \(D\). Thus, if \(k\) is any nonnegative integer, then 
\begin{align*} 
e_i^* A^k e_i = e_i^*\left(UD^k U^*\right)e_i = (U^* e_i)^* D^k (U^* e_i) = \frac{1}{n}\trace D^k = \frac{1}{n}\trace A^k = \tau(A^k),    	
\end{align*}
i.e., \(e_i\) is trace vector for \(A\). The claim now follows from Observation \ref{obs:cancomps}.
\end{proof}  

%-----------------------------------------------------------------------------------------------------------------------------------------------------------------------------------------------------------------------------
\section{The Field of Values.}
%-----------------------------------------------------------------------------------------------------------------------------------------------------------------------------------------------------------------------------

The \emph{field (of values) of \( A \in M_n(\bb{C}) \)}, denoted by \(F(A)\), is defined by \( F(A) = \left\{ x^*A x : x^*x = 1 \right\} \subseteq \bb{C} \). A good general reference for the field is \cite[Chapter 1]{hj1994} and it is sometimes called the \emph{numerical range}. 

Recalling that \( A \in M_n(\bb{C}) \) is \emph{normal} if \( A^*A = AA^*\), the basic properties for us are as follows. 

%-----------------------------------------------------------------------------------------------------------------------------------------------------------------------------------------------------------------------------
\begin{prop} 
\label{fvprops}
If \( A \in M_n(\bb{C}) \), then:
 \begin{enumerate}
 [leftmargin=1cm, label=(\roman*)]
\item $F(A)$ is compact \cite[Property 1.2.1]{hj1994}; 
\item \(F(A)\) is convex \cite[\S 1.3]{hj1994};
\item \label{nrspectrum} \( \sigma(A) \subseteq F(A) \) \cite[Property 1.2.6]{hj1994}; 
\item \label{nrnormal} \( F(A) = \conv{\sigma(A)} \), whenever \(A\) is {normal} \cite[Property 1.2.9]{hj1994}; and
\item \label{nrsubmatrix} \( F(A_{(i)}) \subseteq F(A) \), \( i \in \langle n \rangle \) \cite[Property 1.2.11]{hj1994}.
\end{enumerate}
\end{prop}

If $p \in \mathbb{C}[x,y]$ is a polynomial of degree $n$, then the \emph{plane algebraic curve with respect to p}, denoted by $\gamma$ or $\gamma_p$, is defined by
\[ \gamma = \left\{ (x,y) \in \mathbb{C}^2 \mid p(x,y) = 0 \right\}. \] 
If $P \in \mathbb{C}[x,y,z]$ is a homogeneous polynomial of degree $n$, then the \emph{plane projective curve with respect to P}, denoted by $\kappa$ or $\kappa_P$, is defined by 
\[ \kappa = \left\{ (x,y,z) \in \mathbb{CP}^2 \mid P(x,y,z) = 0 \right\} \] 
(here $\mathbb{CP}^2$ denotes the complex projective plane). The \emph{degree of $\gamma$} (respectively, \emph{degree of $\kappa$}), denoted by $\deg \gamma$ (respectively, $\deg \kappa$), is defined by $\deg \gamma = \deg p$ (respectively, $\deg \kappa = \deg P$). 
The \emph{real part} of a plane algebraic curve $\gamma$ is defined by $\Re(\gamma) = \left\{ (x,y) \in \mathbb{R}^2 \mid p(x,y) = 0 \right\}$. The real part of a plane projective curve is defined similarly.

If $p \in \mathbb{C}[x,y]$, then $H[p](x,y,z) := z^{\deg p} p(x/z,y/z)$ is a homogeneous polynomial. If $P \in \mathbb{C}[x,y,z]$ is a homogeneous polynomial, then $B[P] (x,y) := P(x,y,1)$ is a bivariate (not necessarily homogeneous) polynomial. Thus, every plane algebraic curve can be identified with a plane projective curve and vice versa. 

If $\kappa_P$ is a plane projective curve of degree $n$, then the \emph{dual of $\kappa_P$}, denoted by $(\kappa_P)^\delta$ or $\kappa_{P^\delta}$, is the unique plane projective curve of degree $m$ such that $P^\delta(u, v, w) = 0$ if and only if the line \( ux + vy + wz = 0 \) is tangent to $\kappa_P$. The number $m$ is defined to be the \emph{class} of $\kappa_P$. It is well known that $((\kappa_P)^\delta)^\delta = \kappa_P$; thus, the dual curve has degree $m$ and class $n$. 

A point $Q$ of a plane projective curve is called a \emph{focus} if it is not equal to one of the \emph{circular points} $I = (1, i, 0)$ or $J = (1, -i, 0)$ and the lines $QI$ and $QJ$ are tangent to the curve. For more information on algebraic curves, the interested reader is directed to Salmon \cite{s1960}. 

The following fact is due to Kippenhahn.

%-----------------------------------------------------------------------------------------------------------------------------------------------------------------------------------------------------------------------------
\begin{theorem}
[Kippenhahn {\cite[Theorem 10]{k1951}}]
If $A$ is an $n$-by-$n$ matrix with complex entries, then there is a plane projective curve $\kappa_P$ of class $n$ such that 
\[ F(A) = \conv{\Re(\gamma_{P(x,y,1)})}. \] 
Furthermore, if $H_1 := (A + A^*)/2$ and $H_2 := (A-A^*)/(2i)$, then 
\begin{equation}
P^\delta = \left \vert H_1 u + H_2 v + I_n w \right \vert. 
\end{equation} 
\end{theorem} 

%-----------------------------------------------------------------------------------------------------------------------------------------------------------------------------------------------------------------------------
\begin{corollary}
\label{pwac}
If \( A \in M_n (\bb{C}) \), then the boundary of \(F(A)\) is a piecewise algebraic curve.
\end{corollary}

The following result, called the \emph{elliptical range theorem}, is well known \cite{d1957,hj1994,j1974,l1996,m1932}. 

%-----------------------------------------------------------------------------------------------------------------------------------------------------------------------------------------------------------------------------
\begin{theorem}
[Elliptical range theorem]
\label{ert}
If \( A \in M_2(\bb{C})\) and \( \sigma(A) = \{ \lambda_1, \lambda_2 \}\), then the field of values of \( A \) is a (possibly degenerate) elliptical disk with foci \( \lambda_1\) and \( \lambda_2\), and minor axis \( \sqrt{\trace{(A^*A)} - |\lambda_1|^2 - |\lambda_2|^2} \).
\end{theorem}

%-----------------------------------------------------------------------------------------------------------------------------------------------------------------------------------------------------------------------------
\section{Proofs.}
%-----------------------------------------------------------------------------------------------------------------------------------------------------------------------------------------------------------------------------

First we see that the Gauss--Lucas theorem is an almost immediate consequence of the previous results. 

%----------------------------------------------------------------------------------------------------------------------------------------------------------------------------------------------------------------------------- 
\begin{theorem}
[Gauss--Lucas]
\label{thm:gl}
If $p$ is a polynomial with zeros $\lambda_1,\dots, \lambda_n$ (including multiplicities) and critical points $\mu_1,\dots,\mu_{n-1}$ (including multiplicities), then 
\[ \{ \mu_1,\dots,\mu_{n-1}\} \subseteq \conv{\{ \lambda_1,\dots, \lambda_n \}}. \] 
\end{theorem}

\begin{proof}
Let \( D := \diag{\lambda_1, \dots, \lambda_n}\), $U := F/\sqrt{n}$, and \( A := U D U^*\), in which \( F \) is the DFT matrix of order \(n\). For any \( i \in \langle n \rangle \), notice that, following \hyp{Thereom}{thm:main}, and parts \ref{nrspectrum}, \ref{nrnormal}, and \ref{nrsubmatrix} of \hyp{Proposition}{fvprops}, we obtain 
\begin{equation*}
\{ \mu_1,\dots,\mu_{n-1}\} = \sigma(A_{(i)}) \subseteq F(A_{(i)}) \subseteq F(A) = \conv{\{ \lambda_1,\dots,\lambda_n \}}. \qedhere
\end{equation*}
\end{proof}

If $p$ is a polynomial with real zeros $\lambda_1 \ge \cdots \ge \lambda_n$ (including multiplicities), then not only are the critical points real (as a consequence of Gauss--Lucas), but the critical points $\mu_1 \ge \cdots \ge \mu_{n-1}$ (including multiplicities) must also satisfy 
\[
\lambda_1 \ge \mu_1 \ge \lambda_2 \ge \cdots \ge \lambda_{n-1} \ge \mu_{n-1} \ge \lambda_n,
\]
because the eigenvalues of the Hermitian matrix \( A_{(i)} \) interlace those of the Hermitian matrix $A$ by the Cauchy interlacing theorem \cite[Theorem 4.3.17]{hj2013}. We thus obtain an almost effortless proof of the following classical result (cf., \cite[Theorem 2.1]{f2006}).

%-----------------------------------------------------------------------------------------------------------------------------------------------------------------------------------------------------------------------------
\begin{theorem}
Let \(p\) be a polynomial with real zeros \(\lambda_1,\dots,\lambda_n\) (including multiplicities) and critical points \(\mu_1,\dots,\mu_{n-1}\) (including multiplicities), each listed in descending order. Then the critical points of \(p\) interlace the roots of \(p\), i.e.,
\[
\lambda_1 \ge \mu_1 \ge \lambda_2 \ge \cdots \ge \lambda_{n-1} \ge \mu_{n-1} \ge \lambda_n.
\]
\end{theorem}

%-----------------------------------------------------------------------------------------------------------------------------------------------------------------------------------------------------------------------------
\begin{lemma}
\label{lem:sides}
Let \(A = U D U^* \in M_n(\mathbb{C})\), be normal with $D = \diag{\lambda_1, \dots,\lambda_n}$ and $U$ unitary. Let the polygon $P$ be the boundary of $F(A) = \conv{\lambda_1,\dots,\lambda_n}$. Let $f_A : \mathbb{S}^n \longrightarrow \mathbb{C}$ be such that $x \longmapsto x^* Ax$ (here $\mathbb{S}^n$ is the Euclidean unit sphere in $\mathbb{C}^n$). Let $\lambda_i$ and $\lambda_j$ be adjacent vertices of $P$ such that (i) $\lambda_i$ and $\lambda_j$ are simple eigenvalues of $A$; and (ii) one of the open half-planes determined by the line that passes through $\lambda_i$ and $\lambda_j$ contains the remaining eigenvalues of $A$. If $z$ is on the line segment $\lambda_i \lambda_j$, then $f_A^{-1}(\{ z \}) \subseteq \text{span}(u_i,u_j)$ (here $u_k$ denotes the $k$th column of $U$). 
\end{lemma}

\begin{proof}
Let $x \in f_A^{-1}(\{ z \})$. Since the columns of $U$ form an orthonormal basis, there are unique scalars $c_1,\dots,c_n$ such that 
\[ x = \sum_{k=1}^n c_k u_k \]
and \( u_k u_\ell = \delta_{k\ell}\) (here, $\delta$ denotes the Kronecker delta). 
Since $x$ is a unit vector, a simple calculation reveals that
\[ \sum_{k=1}^n |c_k|^2 = x^*x = 1 .\]

As $U^*u_k = e_k$ for $k \in \{1,\dots,n\}$, we have 
\begin{align*}
z = x^*Ax 
&= \left( \sum_{k=1}^n c_k u_k \right)^* A \left( \sum_{k=1}^n c_k u_k \right) 										\\
&= \left( U^* \sum_{k=1}^n c_k u_k \right)^* D \left(U^*  \sum_{k=1}^n c_k u_k \right) 									\\
&= \left( \sum_{k=1}^n c_k e_k \right)^* D \left( \sum_{k=1}^n c_k e_k \right) 										\\
&= \left( \sum_{k=1}^n \overline{c_k} e_k^\top \right) \left( \sum_{k=1}^n c_k \lambda_k e_k \right) = \sum_{k=1}^n |c_k|^2 \lambda_k.
\end{align*}
Since $\lambda_i$ and $\lambda_j$ are simple eigenvalues of $A$, it follows that $c_k = 0$ whenever $k \ne i,j$ (otherwise, $z$ would be forced off the line segment $\lambda_i \lambda_j$). Thus, $x \in \text{span}(u_i,u_j)$.
\end{proof}

In \cite{s1865} it was shown that the critical points of a polynomial of degree $n$ having distinct zeros are the foci of the curve of class $n - 1$ that is tangent to each line segment joining the zeros of the polynomial at its midpoints. The following result is a weaker version of this result that, in addition to its usefulness, is of interest in its own right.

%----------------------------------------------------------------------------------------------------------------------------------------------------------------------------------------------------------------------------- 
\begin{theorem}
[Poor-man's Siebeck]
\label{pmsiebeck}
Let $p$ be a polynomial of degree $n$ with zeros $\lambda_1,\dots, \lambda_n$ and critical points $\mu_1,\dots,\mu_{n-1}$. Let $C$ denote the convex hull of $\lambda_1, \dots, \lambda_n$ and $P$ be the convex polygon coinciding with the boundary of $C$. Suppose that for every pair of adjacent vertices $\lambda_i$ and $\lambda_j$ of $P$ we have (i) $\lambda_i$ and $\lambda_j$ are simple eigenvalues of $A$; and (ii) one of the open half-planes determined by the line that passes through $\lambda_i$ and $\lambda_j$ contains the remaining eigenvalues of $A$. Then there is a piecewise algebraic curve contained in $C$ that is tangent to the polygon $P$ at the midpoints of its sides and such that $\mu_i$ is contained in the convex hull of the piecewise algebraic curve. Furthermore, the midpoints of the sides of $P$ are the only points of tangency. 
\end{theorem}

\begin{proof}
Let \( D := \diag{\lambda_1, \dots, \lambda_n}\), $U := F/\sqrt{n}$, and \( A := U D U^*\), in which \( F \) is the DFT matrix of order \(n\). Following \hyp{Corollary}{pwac}, the boundary of \( F(A_{(1)})\) is a piecewise algebraic curve and, following \hyp{Theorem}{thm:main} and \hyp{Proposition}{fvprops}, it must be contained in $C$ since \( F(A_{(1)}) \subseteq F(A) = \conv{\lambda_1, \dots, \lambda_n} =C \). Clearly, the critical points are contained in the convex hull of $C$. 

Next, we show that the boundary of \( F\left( A_{(1)} \right)\) is tangent to $P$ at the midpoints of its sides. To that end, suppose that $\lambda_i$ and $\lambda_j$ are adjacent vertices of $P$. If \( u_k \) denotes the \(k\){th} column of \(U\), then \( u_k^* u_\ell = \delta_{k \ell}\). 
If \( v := (u_i - u_j)/\sqrt{2}\),  
then 
\[ \begin{Vmatrix} v \end{Vmatrix}_2^2 = \frac{(u_i - u_j)^*(u_i - u_j)}{2}  = \frac{u_i^*u_i + u_j^*u_j}{2} = 1, \] 
and \( Av = (\lambda_i u_i - \lambda_j u_j)/\sqrt{2}\).   
The first entry of \(v\) is zero since the first row of \(F\) is the all-ones vector; thus, \( v_{(1)} \) is a unit vector, \( v_{(1)}^* A_{(1)} v_{(1)} \in F(A_{(1)}) \),
and 
\begin{align*}
v_{(1)}^* A_{(1)} v_{(1)} = v^* Av 
= \frac{\left( u_i^* - u_j^* \right) \left( \lambda_i u_i - \lambda_j u_j \right)}{2} 	
&= \frac{\lambda_i u_i^* u_i + \lambda_j u_j^* u_j}{2} 					\\
&= \frac{\lambda_i + \lambda_j}{2}, 
\end{align*}
i.e., the midpoint of the line segment \( \lambda_i \lambda_j\) belongs to both \( F(A_{(1)}) \) and \( F(A) \). 

We now show that the midpoint $z := v_{(1)}^*A_{(1)} v_{(1)}$ is the only point of tangency of the line segment $\lambda_i \lambda_j$. Notice that for any point $x \in F(A_{(1)})$ there is a unit vector $w$ whose first entry is zero such that $x=w^*Aw$. If $x$ is on the line segment $\lambda_i\lambda_j$, then $x$ must be in the span of $u_i$ and $u_j$ by Lemma \ref{lem:sides}. Since the first component of $w$ is zero, it follows that $w = v$, i.e., $x = z$. 
\end{proof}

%----------------------------------------------------------------------------------------------------------------------------------------------------------------------------------------------------------------------------- 
\begin{remark}
\hyp{Figure}{fig:pms} illustrates \hyp{Theorem}{pmsiebeck}. 
\end{remark}

\begin{figure}[H]
\centering
\includegraphics[width=.75\linewidth]{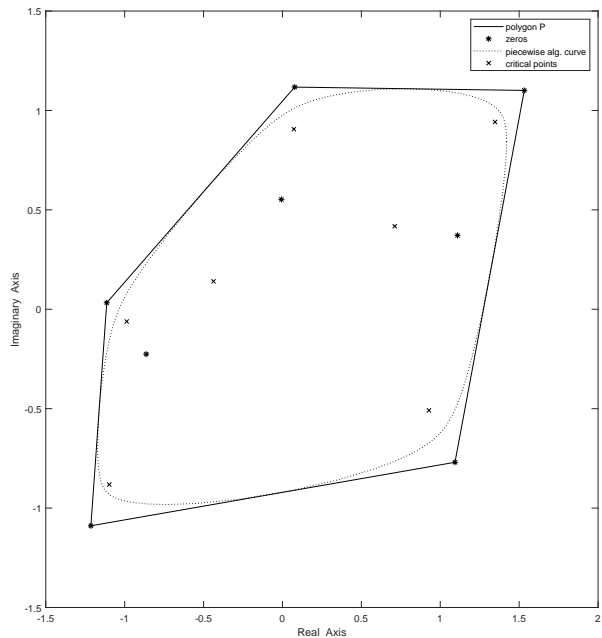}
\caption{A randomly generated example in MATLAB illustrating Theorem \ref{pmsiebeck}.}
\label{fig:pms}
\end{figure}

The B\^{o}cher--Grace--Marden theorem is now an immediate consequence of \hyp{Theorem}{ert} amd \hyp{Theorem}{pmsiebeck}.

%-----------------------------------------------------------------------------------------------------------------------------------------------------------------------------------------------------------------------------
\begin{theorem}
[B\^{o}cher--Grace--Marden]
\label{bgm}
If $p$ is a polynomial of degree three with noncollinear zeros \( \lambda_1, \lambda_2,\) and \(\lambda_3\), then the critical points of $p$ are the foci of the unique ellipse inscribed in the triangle with vertices \( \lambda_1, \lambda_2,\) and \(\lambda_3\) tangent to its sides at their midpoints. 
\end{theorem}

%-----------------------------------------------------------------------------------------------------------------------------------------------------------------------------------------------------------------------------
\begin{remark}
The ellipse described in \hyp{Theorem}{bgm} is called the \emph{Steiner inellipse}. For more information, and a proof of its uniqueness, see, e.g., Steiner \cite{s1882} or Kalman \cite{k2008}. 

\hyp{Figure}{fig:bgm} contains an example illustrating Theorem \ref{bgm}.
\end{remark}

\begin{figure}[H]
\centering
\includegraphics[width=.75\linewidth]{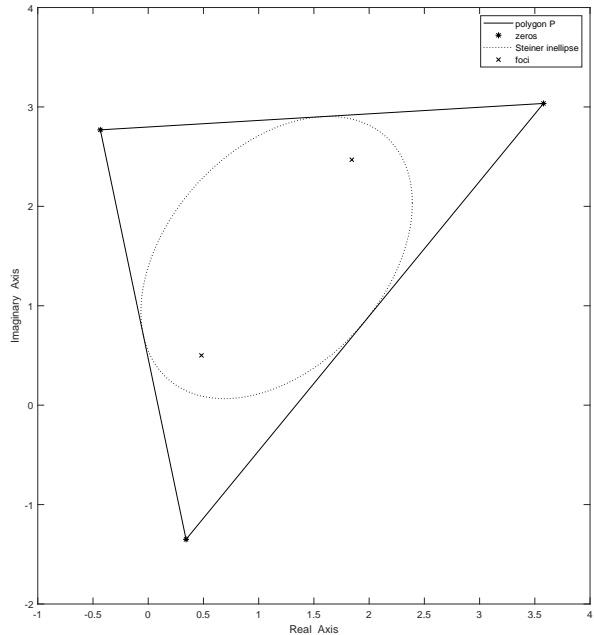}
\caption{A randomly generated example in MATLAB illustrating Theorem \ref{bgm}.}
\label{fig:bgm}
\end{figure}

%-----------------------------------------------------------------------------------------------------------------------------------------------------------------------------------------------------------------------------
\section{Concluding Remarks.}
%-----------------------------------------------------------------------------------------------------------------------------------------------------------------------------------------------------------------------------

Of interest would be a proof of the full version of Siebeck's theorem with the methods outlined above.

%----------------------------------------------------------------------------------------------------------------------------------------------------------------------------------------------------------------------------- 
\begin{theorem}
[Siebeck \cite{s1865}]
\label{siebeckfull}
If \(\lambda_1, \dots, \lambda_n \in \mathbb{C}\) are distinct and 
\[ p(t) = \alpha \prod_{i=1}^n (t - \lambda_i)^{m_i} \in \bb{C}[t],~\alpha \ne 0, \]  
then the critical points of \(p\) are the foci of the curve of class \( n - 1 \) that touches each line segment \( \lambda_i \lambda_j \) in a point dividing the line segment in the ratio \( m_i:m_j \).  
\end{theorem}

%-----------------------------------------------------------------------------------------------------------------------------------------------------------------------------------------------------------------------------
\subsection{Authors' Note.}
%-----------------------------------------------------------------------------------------------------------------------------------------------------------------------------------------------------------------------------

After this manuscript was submitted, it was brought to our attention that another \textsc{Monthly} article \cite{gork} also provided a matricial proof of the B\^{o}cher--Grace--Marden theorem (Theorem \ref{bgm}). Both provide a more efficient explanation of B\^{o}cher--Grace--Marden than \cite{k2008}, which was intended to clarify and put B\^{o}cher--Grace--Marden on a rigorous footing (using only analysis and not matrices). Both the present proof and that of \cite{gork} use the field of values and principal submatrix containment. Otherwise, however, they are very different. By producing a normal matrix from the polynomial roots, using the DFT matrix and some other simple, but not so well known, matricial ideas, our proof is much shorter and may make this surprising fact even more transparent. 

%-----------------------------------------------------------------------------------------------------------------------------------------------------------------------------------------------------------------------------
\begin{acknowledgment}{Acknowledgment.}
The authors thank the anonymous referees and editor-in-chief Susan Colley for their helpful comments. 
\end{acknowledgment}

%-----------------------------------------------------------------------------------------------------------------------------------------------------------------------------------------------------------------------------
% Bibliography 

\begin{biog}
\item[Charles R.~Johnson] graduated from Elkhart (IN) High School in 1966, from Northwestern University, with a degree in Mathematics and Economics, in 1969, and then received his Ph.D. from the California Institute of Technology in 1972. After an NRC/NAS postdoc at the National Bureau of Standards, he took a joint position in the Institute for Physical Science and Technology and the Department of Economics at the University of Maryland in 1974, where he was tenured in 1976. After two years as Professor of Mathematical Sciences at Clemson University, he took the Class of 1961 Professorship of Mathematics at William and Mary in 1987. He has now published well over 400 papers and several books, including \emph{Matrix Analysis} and \emph{Topics in Matrix Analysis} (with Roger Horn) and \emph{Totally Nonnegative Matrices} (with Shaun Fallat). His most recent book, \emph{Eigenvalues, Multiplicities and Graphs} (with Carlos Saiago), just appeared from Cambridge University Press. Much of his work is in matrix theory and combinatorics, but he has papers in journals of physics, economics, psychology, finance, and statistics, etc. He has won awards such as the Washington Academy of Sciences Award for Outstanding Scientific Achievement and the Virginia Outstanding Faculty Member Award and been editor of several journals. He continues to run a long-standing REU program, for which he welcomes applications from any students deeply interested in mathematics. This has resulted in several dozen publications in high-level journals.
\begin{affil}
Department of Mathematics, College of William \& Mary, Williamsburg, VA 23187-8795, USA\\
crjohn@uw.edu
\end{affil}

\item[Pietro Paparella] received the Ph.D. degree in mathematics from Washington State University in 2013 under the supervision of Michael Tsatsomeros and Judi McDonald. From 2013 to 2015 he held the position of Visiting Assistant Professor in the Department of Mathematics at the College of William and Mary and since 2015 he has held the position of Assistant Professor in the Division of Engineering and Mathematics at the University of Washington Bothell. His research interests are in nonnegative matrix theory, combinatorial matrix theory, discrete geometry, and the geometry of polynomials. This work is his first in the \emph{American Mathematical Monthly}.  
\begin{affil}
Division of Engineering and Mathematics, University of Washington Bothell, Bothell, WA 98011, USA\\
pietrop@uw.edu
\end{affil}
\end{biog}
\vfill\eject

\end{document}